\documentclass[12pt]{amsart}
\usepackage{amsfonts}
\usepackage{amsmath}
\usepackage{amssymb}
\usepackage{graphicx}
\usepackage{enumerate}
\usepackage{multicol}
\usepackage{mathrsfs}
\usepackage[usenames,dvipsnames]{pstricks}
\usepackage{epsfig}
\usepackage{pst-grad} 
\usepackage{pst-plot} 
\usepackage[hmargin=3cm,vmargin=2.5cm]{geometry}

\usepackage[utf8]{inputenc}

\newtheorem{theorem}{Theorem}[section]
\newtheorem{prop}[theorem]{Proposition}

\newtheorem{coro}[theorem]{Corollary}
\newtheorem{lemma}[theorem]{Lemma}
\newtheorem{defi}[theorem]{Definition}

\title{Rolling against a sphere: The non transitive case}
\author[Y. Chitour, M. Godoy M., P. Kokkonen, I. Markina]{Yacine Chitour\\Mauricio Godoy Molina\\Petri Kokkonen\\Irina Markina}

\address{L2S, Universit\'e Paris-Sud XI, CNRS and Sup\'elec, Gif-sur-Yvette, 91192, France. }
\email{yacine.chitour@lss.supelec.fr}

\address{Department of Mathematics, University of Bergen, Norway.}
\email{mauricio.godoy@math.uib.no}

\address{Varian Medical Systems, Helsinki, Finland.}
\email{pvkokkon@gmail.com}

\address{Department of Mathematics, University of Bergen, Norway.}
\email{irina.markina@math.uib.no}

\thanks{The second author is partially supported by the grant of the Norwegian Research Council 213440/BG. The fourth author is partially supported by the grant of the Norwegian Research Council 204726/V30.}

\subjclass[2000]{53C05, 53C29, 70G45}

\keywords{Rolling system, Sasakian manifold, 3-Sasakian manifold, nonholonomic mechanics, holonomy of connections}

\begin{document}

\maketitle

\begin{abstract}
We study the control system of a Riemannian manifold $M$ of dimension $n$ rolling on the sphere $S^n$. The controllability of this system is described in terms of the holonomy of a vector bundle connection which, we prove, is isomorphic to the Riemannian holonomy group of the cone $C(M)$ of $M$. 

Using Berger's list, we reduce the possible holonomies to a few families. In particular, we focus on the cases where the holonomy is the unitary and the symplectic group. In the first case, using the rolling formalism, we construct explicitly a Sasakian structure on $M$; and in the second case, we construct a 3-Sasakian structure on $M$.
\end{abstract}

\section{Introduction}

It is well known, that modern control theory began mostly as a linear theory, having its roots in electrical engineering by making use of linear algebra, complex and functional analysis as its main tools. Nonlinear control theory relies to a large extend on differential geometry. In the present paper we deal with a generalization of the classical nonlinear mechanical problem: rolling of rigid bodies. More precisely, we consider  an $n$-dimensional Riemannian manifold $M$ rolling on a space form of positive curvature without slipping and twisting. Our main interest is to relate the controllability of this system with the geometric properties of the rolling manifold $M$.  The kinematic constraints of no-slipping and no-spinning have nonholonomic nature. Core concepts for studying the control and geometry of nonholonomic systems are the notions of fiber bundle and associated connections. The bundle point of view not only gives us a way of organizing variables in a physically meaningful way, but reveals the basic idea of the relation between geometric properties of the two rolling manifolds. A bundle connection relates base and fiber variables in the system, and in this sense one can take a gauge theoretical point of view of nonholonomic control systems.

Thus the rolling system of the manifold $M$ over the space form of positive curvature, that is, the $n$-dimensional sphere of some radius, constrained by the no-slipping and no-twisting conditions, leads to the principal $SO(n+1)$-bundle $Q\to M$, where $Q$ is the configuration space of the rolling system. The rolling distribution, encoding the no-slipping and no-twisting conditions, is an Ehresmann connection. We associate a vector bundle to the $SO(n+1)$-bundle $Q\to M$, by making use of representation theory. This allows us to use a metric connection related to the vector bundle, that is naturally called the rolling connection. In the next step we show that the holonomy group of this rolling connection is isomorphic to the Riemannian holonomy of the Riemannian cone $C(M)$ over the original manifold $M$. The details are presented in Sections~2 and 3. There we also reveal the relation between the properties of the holonomy group of the Riemannian cone $C(M)$ and the geometry of the manifold $M$. We would like to stress that the holonomy group of the rolling connection completely determines the controllability of the rolling system. Section 4 shows that if the holonomy group of the rolling connection for an odd dimensional manifold $M$ is a subgroup of the unitary group, then the rolling system is not controllable and moreover the manifold $M$ itself inherits the Sasakian structure of odd dimensional sphere over which the manifold $M$ is rolled. We restore all the Sasakian structure on $M$, starting from the holonomy group of the rolling connection. The Heisenberg group gives an example of the manifold $M$ whose rolling holonomy group is isomorphic to the unitary group. The last Section 5 consider the case of the symplectic holonomy groups and their relation to the 3-Sasakian structure of the rolling manifold $M$. 

\section{Preliminaries}

We assume that all the Riemannian manifolds we will be working with are connected, simply connected, oriented and complete.

\subsection{The rolling system}

The higher dimensional (intrinsic) rolling system has been introduced in~\cite{CK1,GGML}. Intuitively, this abstract mechanical system consists of two $n$ dimensional Riemannian manifolds $(M,g)$ and $(\widehat M,\widehat g)$ rolling one on the other. More precisely, we can construct the {\em state space}
\[
Q=Q(M,\widehat M)=\{A\colon T_xM\to T_{\widehat x}\widehat M\,|\,A\in{\rm SO}(T_xM,T_{\widehat x}\widehat M),x\in M,\widehat x\in\widehat M\}.
\]
An absolutely continuous curve $q\colon[0,\tau]\to Q$ is called a {\em rolling curve} if $q(t)=(x(t),\widehat x(t);A(t))$ satisfies the conditions
\begin{itemize}
\item[(NS)] $\dot{\widehat x}(t)=A(t)\dot x(t)$, for almost all $t\in[0,\tau]$.

\item[(NT)] $q(t) \frac{D}{dt} Z(t) = \frac{D}{dt} q(t) Z(t)$ for any tangent vector field $Z(t)$ along $x(t)$, for almost all $t\in[0,\tau]$, where $\frac{D}{dt}$ denotes the covariant derivative on $M$ or $\widehat M$.
\end{itemize}
The restriction (NS) is often called the {\em no-slip} condition, and (NT) the {\em no-twist} condition.

The rolling curves determine a rank $n$ subbundle ${\mathcal D}_{\rm R}\hookrightarrow TQ$, called the {\em rolling distribution} defined as
\[
{\mathcal D}_{\rm R}|_q=\{\dot q(0)\,|\,q\mbox{ is a rolling curve s.t. }q(0)=q\},\quad q\in Q.
\]
Alternatively, we can define this distribution by means of the {\em rolling lift} ${\mathscr L}_{\rm R}(X)|_q\in T_qQ$ of a tangent vector $X\in T_xM$, given by
\[
{\mathscr L}_{\rm R}(X)|_q=\left.\frac{d}{dt}\right|_{t=0}\Big(P^t_0(\widehat\gamma)\circ A\circ P^0_t(\gamma)\Big),
\]
where $q=(x,\widehat x;A)$, and $(\gamma,\widehat\gamma)$ is any pair of curves satisfying 
\[
(\gamma(0),\widehat\gamma(0))=(x,\widehat x),\quad(\dot\gamma(0),\dot{\widehat\gamma}(0))=(X,AX).
\]
As usual $P_a^b(\gamma)$ (resp. $P_a^b(\widehat\gamma)$) denotes the parallel transport in $M$ along $\gamma$ from $\gamma(a)$ to $\gamma(b)$ (resp. $\widehat M$ along $\widehat\gamma$ from $\widehat\gamma(a)$ to $\widehat\gamma(b)$). In this terms we have that
\[
{\mathcal D}_{\rm R}|_q={\mathscr L}_{\rm R}(T_xM)|_q,
\]
and a curve $q\colon[0,\tau]\to Q$ is a rolling curve if and only if $\dot q(t)\in{\mathcal D}_{\rm R}|_{q(t)}$ for almost all $t\in[0,\tau]$. For more details regarding the rolling lift, see~\cite{K1}, and for a coordinate description of ${\mathcal D}_{\rm R}$, see~\cite{GGML}.

The controllability question for the rolling system asks whether for any two given points $q_0,q_1\in Q$ there exists a rolling curve $q\colon[0,\tau]\to Q$ such that $q(0)=q_0$ and $q(\tau)=q_1$. This question is in general very hard to answer, and thus we focus on the special case when the manifold $\widehat M$ has some simpler geometry.

\subsection{Rolling against a space form}\label{ssec:spform}

In~\cite{CK2,CGK2}, the authors study the rolling system when $(\widehat M,\widehat g)$ is a space form, that is, a complete and simply connected Riemannian manifold of constant sectional curvature $c\in {\mathbb R}$. In this case, the natural projection $\pi_{Q,M}\colon Q\to M$ is a principal bundle of a special form, which we proceed to explain.

Recall that on a $G$-principal bundle $\pi\colon P\to M$ a $G$-invariant subbundle $D\subset TP$ is a {\em horizontal distribution} if it satisfies $T_pP=D_p\oplus \ker d_p\pi$ for any $p\in P$. 

\begin{theorem}\label{th:ppal}
The projection $\pi_{Q,M}\colon Q\to M$ admits a $G$-principal bundle structure with horizontal distribution ${\mathcal D}_{\rm R}$, for some Lie group $G$, if and only if $\widehat M$ is a space form (under some genericity assumptions, see~\cite[Theorem 4.10]{CGK2}). In this situation, the Lie group $G$ is given by
\[
G=\begin{cases}{\rm SO}(n+1),&c>0,\\{\rm SE}(n),&c=0,\\{\rm SO}_0(n,1),&c<0,\end{cases}
\]
where $c$ is the curvature of $\widehat M$.
\end{theorem}

According to~\cite[Proposition 2.3.7]{J}, it is possible to associate to a principal bundle a vector bundle via a representation of its structure group. In the non-Euclidean case $c\neq0$, using Theorem~\ref{th:ppal} and the canonical representation of $G$ on ${\mathbb R}^{n+1}$, we obtain the so-called {rolling connection} on the vector bundle $\pi\colon TM\oplus{\mathbb R}\to M$. 

\begin{defi}\label{def:conn}
The rolling connection $\nabla^{{\rm R},c}$, is a connection for the vector bundle $\pi\colon TM\oplus{\mathbb R}\to M$, given by
\[
\nabla^{{\rm R},c}_Y(X,r)=\Big(\nabla_YX+r(x)Y,Y(r)-cg(X_x,Y)\Big),
\]
where $x\in M$, $Y\in T_xM$, $X\in{\rm VF}(M)$ and $r\in C^\infty(M)$. 
\end{defi}

We therefore obtain a holonomy $H^{\nabla^{{\rm R},c}}\subset{\rm GL}(n+1)$ called the {\em rolling holonomy} group. Moreover, this connection is metric with respect to the fiber inner products
\begin{equation}\label{eq:hc}
h_c\big((X,r),(Y,s)\big)=g(X,Y)+c^{-1}rs.
\end{equation}
If $c<0$, then $H^{\nabla^{{\rm R},c}}\subset {\rm SO}_0(n,1)$, and $H^{\nabla^{{\rm R},c}}\subset{\rm SO}(n+1)$ when $c>0$. 


%

In what follows, we will normalize the value of the curvature $c\in{\mathbb R}$ to $c\in\{-1,1\}$. With this extra geometric information of the rolling system in the case $\widehat M$ is a space form, it is possible to give general answers to the controllability question. Using Theorem~\ref{th:ppal} and~\cite[Proposition 2.3.7]{J}, it is not difficult to see that the rolling system being controllable is equivalent to determining whether the system has full rolling holonomy group, that is, whether $H^{\nabla^{{\rm R},c}}=G$ in the theorem above. For the hyperbolic case $c=-1$, a full answer was given in~\cite{CGK1} by means of the following result

\begin{theorem}
For $c=-1$, the rolling system is not controllable if and only if there exists a complete simply connected Riemannian manifold $(M_1,g_1)$ such that $(M,g)$ is a warped product either of the form
\begin{itemize}
\item[{\rm (WP1)}] $({\mathbb R}\times M_1, ds^2\oplus_{e^{-s}} g_1)$, or

\item[{\rm (WP2)}] $({\mathbb H}^k\times M_1,{\bf g}^k\oplus_{\cosh(d)} g_1)$, where $1\leq k\leq n$, ${\bf g}^k$ is the usual hyperbolic metric and for each $x\in{\mathbb H}^k$, $d(x)$ is the distance between $x$ and an arbitrary fixed point $x_0\in{\mathbb H}^k$. 
\end{itemize}
\end{theorem}

In the spherical case, that is $c=1$, the situation is more complicated due to the non-trivial topology of spheres, see~\cite{CK2}. In fact, there is an almost complete answer in the case of even dimensional spheres. Recall that since $H^{\nabla^{{\rm R},1}}$ is a subgroup of ${\rm SO}(n+1)$, there is a natural action of $H^{\nabla^{{\rm R},1}}$ on the unit sphere $S^n$. In this framework, the following result holds, see~\cite{CK2}.

\begin{theorem}\label{th:CKSn}
Assume that the action of $H^{\nabla^{{\rm R},1}}$ on the unit sphere is not transitive, then $(M,g)$ admits the unit sphere as its universal covering space. In particular, in this situation the rolling distribution ${\mathcal D}_{\rm R}$ is involutive and, thus, the rolling system is not controllable.
\end{theorem}

Therefore, in order to understand the controllability of the system, one needs to have a complete knowledge of the subgroups of ${\rm SO}(n+1)$ acting transitively on the unit sphere. The following is a classical result answering this question.

\begin{theorem}[Montgomery-Samelson~\cite{MS}]\label{th:MS}
Let $H$ be a connected (compact) subgroup of ${\rm SO}(n+1)$ which acts transitively on the unit sphere $S^n\subset{\mathbb R}^{n+1}$. Then $H$ is one of the following groups:
\begin{enumerate}
\item ${\rm SO}(n+1)$,

\item ${\rm U}(m)$, if $n=2m-1$,

\item ${\rm SU}(m)$, if $n=2m-1$, 

\item ${\rm Sp}(k)$, if $n=4k-1$,

\item ${\rm Sp}(k)\cdot{\rm Sp}(1)$, if $n=4k-1$,

\item ${\rm Sp}(k)\cdot{\rm U}(1)$, if $n=4k-1$,

\item $G_2$, $n=6$,

\item ${\rm Spin}(7)$, $n=7$,

\item ${\rm Spin}(9)$, $n=15$.
\end{enumerate}
\end{theorem}
As is usual in the holonomy literature, the notation $A\cdot B$ stands for $(A\times B)/{\mathbb Z}_2$. In both cases above where this notation appears, the subgroup ${\mathbb Z}_2$ can be seen in the standard real matrix representation simply as $\{\pm{\rm Id}\}$.

As a consequence of Theorems~\ref{th:CKSn} and~\ref{th:MS}, we have a characterization of the controllability of the system of $M$ rolling on $S^n$ in terms of the geometry of $M$, for almost all $n$ even, which improves the result~\cite[Corollary 4.7]{CK2}.


\begin{coro}
If $n$ is even and $n\neq6$, then the rolling system is completely controllable if and only if $(M,g)$ is not of constant curvature~1.
\end{coro}

Therefore, to complete the picture for the case of $\widehat M$ a space form, it remains to study the controllability problem in the following situations:
\begin{description}
\item[(Even)] The manifold $M$ (and thus the sphere) has dimension $6$.

\item[(Odd)] The manifold $M$ is odd dimensional.
\end{description}

The aim of this paper is to start the study of the case {\bf (Odd)}, when the action of the rolling holonomy group $H^{\nabla^{{\rm R},c}}$ on $S^n$ is transitive. 

\section{Holonomy of cones and rolling holonomy}

\subsection{Sasakian, Einstein-Sasakian and 3-Sasakian manifolds}

The aim of this subsection is to present the definition of the manifolds of our concern. For a full and detailed study of the geometry and topology of these manifolds, we refer the interested reader to~\cite{BG}. We will employ the canonical identification of $(1,1)$-tensor fields with fields of endomorphisms.

An {\em almost contact structure} on a manifold $M$ is a triplet $(\xi,\eta,\Phi)$, where $\eta\in\Omega^1(M)$ is a one-form, $\xi\in{\rm VF}(M)$ is a vector field, and $\Phi$ is a $(1,1)$-tensor field satisfying
\[
\eta(\xi)=1\quad\mbox{and}\quad\Phi^2=-{\rm id}_{TM}+\xi\otimes\eta.
\]
An almost contact structure $(\xi,\eta,\Phi)$ on a manifold $M$ is a {\em contact structure} if the one form $d\eta(\xi,\cdot)$ vanishes identically. In this situation, $\xi$ is called {\em Reeb vector field} and $\eta$ is called {\em contact form}.

A {\em contact metric structure} on a manifold $M$ is a quadruplet $(\xi,\eta,\Phi,g)$, where the triplet $(\xi,\eta,\Phi)$ is a contact structure and $g$ is a Riemannian metric satisfying the condition
\[
g(\Phi X,\Phi Y)=g(X,Y)-\eta(X)\eta(Y),
\]
for all vector fields $X,Y\in{\rm VF}(M)$. We say that the tensor field $\Phi$ is {\em compatible} with respect to the metric $g$.

Recall that the {\em cone} $(C(M),\tilde g)$ of a Riemannian manifold $(M,g)$ is the warped product
\[
C(M)=M\times{\mathbb R}_+,\quad\tilde g=r^2g+dr^2.
\]
It is well-known, see~\cite[Proposition 6.5.2]{BG}, that there is a one-to-one correspondence between the set of almost contact structures $(\xi,\eta,\Phi)$ on $M$ and almost complex structures $I$ on $C(M)$ satisfying some natural conditions.

A {Sasakian structure} on a manifold $M$ is a contact metric structure $(\xi,\eta,\Phi,g)$ on $M$, where the corresponding almost complex structure $I$ on $C(M)$ is integrable.
 
\begin{defi}
A Riemannian manifold $(M,g)$ with a Sasakian structure $(\xi,\eta,\Phi,g)$ is called a {Sasakian manifold}, and the Reeb vector field $\xi$ is called the {characteristic vector field}.
\end{defi}

A {\em Sasaki-Einstein manifold} $M$ is a Sasakian manifold which is also Einstein, that is, its Ricci tensor ${\rm Ric}_g$ is a constant multiple of the metric $g$. It follows from some general considerations, see~\cite{S}, that in the case that the dimension of $M$ is $2m+1$, the Einstein equation takes the form
\[
{\rm Ric}_g=2mg.
\]

In a similar form to the definition of Sasakian structure, we can define a ``quaternionic'' analogue, based on \cite[Proposition 1.2.2]{BG1}. 

\begin{defi}\label{def:3sas}
A manifold $M$ is {3-Sasakian} if it admits three Sasakian structures with characteristic vector fields $\xi_1,\xi_2,\xi_3$ which are orthonormal and satisfy $[\xi_a,\xi_b]=2\epsilon_{a,b,c}\xi_c$ for $\{a,b,c\}=\{1,2,3\}$, where $\epsilon_{a,b,c}$ is the sign of the permutation $\begin{pmatrix}1&2&3\\a&b&c\end{pmatrix}$.
\end{defi}

\subsection{Comparison of holonomies}

The following table contains the equivalent formulation of several geometric structures over a Riemannian manifold $N$ of dimension $n$ in terms of its Riemannian holonomy $H^{\nabla}$. Some of these are classic, and the reader can consult~\cite{B,BG} for details.

\begin{center}
\begin{table}[h]
\caption{Geometric structures in terms of holonomy.}\label{table:hol}
\begin{tabular}{|c|c|c|}
\hline
Dimension of $N$&Name&Riemannian Holonomy
\\
\hline \hline
$n=2k$&K\"ahler&$H^{\nabla}\subset{\rm U}(k)$ \\
\hline
$n=2k$&Calabi-Yau&$H^{\nabla}\subset{\rm SU}(k)$ \\
\hline
$n=4k$&Hyper-K\"ahler&$H^{\nabla}\subset{\rm Sp}(k)$ \\
\hline
$n=4k$&Quaternionic K\"ahler&$H^{\nabla}\subset{\rm Sp}(k)\cdot{\rm Sp}(1)$ \\
\hline
\end{tabular}
\end{table}
\end{center}

It is well-known that the holonomy of $(C(M),\tilde g)$ encodes important geometric information about the manifold $M$. Let us briefly recall some well-known equivalences that can be found, for instance, in~\cite{BG}.

\begin{theorem}\label{th:holcone}
The following equivalences hold:
\begin{itemize}
\item $(M,g)$ is a Sasakian manifold if, and only if, $(C(M),\tilde g)$ is K\"ahler.

\item $(M,g)$ is a Sasaki-Einstein manifold if, and only if, $(C(M),\tilde g)$ is Calabi-Yau.

\item $(M,g)$ is a 3-Sasakian manifold if, and only if, $(C(M),\tilde g)$ is hyper-K\"ahler.
\end{itemize}
\end{theorem}

Before proceeding with one of the main results of the present paper, let us state a technical lemma.

\begin{lemma}\label{lem:covder}
Let $X,Y\in{\rm VF}(M)$, $b\in C^\infty(M)$ and $(x,s)\in C(M)=M\times{\mathbb R}_+$. Then
\begin{align*}
\tilde\nabla_X|_{(x,s)}(Y+b\partial_s)&=\nabla_X|_xY+\frac{b(x)}{s}X+\big(X|_x(b)-sg(X|_x,Y|_x)\big)\partial_s|_{(x,s)},\\
\tilde\nabla_{\partial_s}|_{(x,s)}(Y+b\partial_s)&=\frac1{s}Y|_x,
\end{align*}
where $\tilde\nabla$ and $\nabla$ denote the Levi-Civita connection on  the cone $(C(M),\tilde g)$ and on $(M,g)$, respectively.
\end{lemma}

\begin{proof}
As noted in~\cite[p.206]{ON}, from the Koszul formula it follows that on a warped product of the form $(C(M),\tilde g)$ we have that 
\begin{align*}
\tilde\nabla_X|_{(x,s)}\partial_s&=\tilde\nabla_{\partial_s}|_{(x,s)}X=\frac1{s}\partial_s|_{(x,s)},\\
\tilde\nabla_X|_{(x,s)}Y&=\nabla_X|_xY-sg(X|_x,Y|_x)\partial_s|_{(x,s)},\\
\tilde\nabla_{\partial_s}|_{(x,s)}\partial_s&=0.
\end{align*}
The formulas in the statement follow from these and the standard calculation rules for affine connections.
\end{proof}

The following result relates the Riemannian holonomy $H^{\tilde\nabla}$ of $(C(M),\tilde g)$ with the rolling holonomy obtained by rolling $M$ over the unit sphere. We denote the rolling connection $\nabla^{{\rm R},1}$ by $\nabla^{\rm R}$, and the corresponding holonomy by $H^{\nabla^{\rm R}}$.

\begin{theorem}\label{th:hol}
Let $(M,g)$ be a Riemannian manifold. We have for every $(x, s) \in C(M)$ that $H^{\tilde\nabla}|_{(x,s)}$ is isomorphic to $H^{\nabla^{\rm R}} |_x$.
More precisely, defining
\[
\begin{array}{ccccc}{\mathcal I}_{(x,s)}&\colon&T_{(x,s)}C(M)&\to&T_xM\times{\mathbb R}\\&&X+b\partial_s|_{(x,s)}&\mapsto&(sX,b)\end{array}
\]
then for all loops $\Gamma(t) = (\gamma(t), a(t))$ of $C(M)$ based at $(x, s)$ one has
\[
{\mathcal I}_{(x,s)}\circ \big(P^{\tilde\nabla}\big)_0^1(\Gamma)=\big(P^{\nabla^{\rm R}}\big)_0^1(\gamma)\circ{\mathcal I}_{(x,s)}.
\]
\end{theorem}

\begin{proof}
Let $Y(t)+b(t)\partial_s|_{\Gamma(t)}$ be a vector field on $C(M)$ along a curve $\Gamma(t)=(\gamma(t),a(t))$. By definition $Y(t)$ is a vector field on $M$ along $\gamma$, and  $\dot\Gamma(t)=\big(\dot\gamma(t),\dot a(t)\partial_s|_{\Gamma(t)}\big)$. Using the formulas in Lemma~\ref{lem:covder}, we see that
\begin{align*}
\tilde\nabla_{\dot\Gamma(t)}(Y(t)+b(t)\partial_s)&=\nabla_{\dot\gamma(t)}Y(t)+\frac{b(t)}{a(t)}\dot\gamma(t)+\\
&+\big(\dot b(t)-a(t)g(\dot\gamma(t),Y(t))\big)\partial_s|_{\Gamma(t)}+\frac{\dot a(t)}{a(t)}Y(t)\\
&=\frac1{a(t)}\big(\nabla_{\dot\gamma(t)}(a(t)Y(t))+b(t)\dot\gamma(t)\big)+\\
&+\big(\dot b(t)-g(\dot\gamma(t),a(t)Y(t))\big)\partial_s|_{\Gamma(t)}\big)\\
&={\mathcal I}_{\Gamma(t)}^{-1}\big(\nabla^{\rm R}_{\dot\gamma(t)}(aY,b)\big)={\mathcal I}_{\Gamma(t)}^{-1}\big(\nabla^{\rm R}_{\dot\gamma(t)}({\mathcal I}_{\Gamma(\cdot)}(Y+b\partial_s))\big).
\end{align*}

Suppose now that $Y(t)+b(t)\partial_s|_{\Gamma(t)}$ is a parallel vector field along $\Gamma$. We obtain the following equations from the discussion above
\begin{align*}
Y(t)+b(t)\partial_s|_{\Gamma(t)}&=\big(P^{\tilde\nabla}\big)_0^t(\Gamma)\big(Y(0)+b(0)\partial_s|_{\Gamma(0)}\big)\\
{\mathcal I}_{\Gamma(t)}\big(Y(t)+b(t)\partial_s|_{\Gamma(t)}\big)&=\big(P^{\nabla^{\rm R}}\big)_0^t(\gamma)\big(a(0)Y(0),b(0)\big).
\end{align*}
Therefore
\begin{multline*}
\Big({\mathcal I}_{\Gamma(t)}\circ\big(P^{\tilde\nabla}\big)_0^t(\Gamma)\Big)\big(Y(0)+b(0)\partial_s|_{\Gamma(0)}\big)=\big(P^{\nabla^{\rm R}}\big)_0^t(\gamma)\big(a(0)Y(0),b(0)\big)\\
=\Big(\big(P^{\nabla^{\rm R}}\big)_0^t(\gamma)\circ{\mathcal I}_{\Gamma(0)}\Big)\big(Y(0)+b(0)\partial_s|_{\Gamma(0)}\big).
\end{multline*}

Since $Y(0)$ and $b(0)$ are arbitrary, we have that the equality
\[
{\mathcal I}_{\Gamma(t)}\circ\big(P^{\tilde\nabla}\big)_0^t(\Gamma)=\big(P^{\nabla^{\rm R}}\big)_0^t(\gamma)\circ{\mathcal I}_{\Gamma(0)}
\]
holds for all $t\in[0,1]$. Finally, assuming that $\Gamma(t)$ is a loop in $C(M)$ and taking $t=1$ in the formula above, we obtain the desired result.
\end{proof}

Two rather immediate consequences of Theorems~\ref{th:holcone} and~\ref{th:hol} that are relevant for the rolling system are the following.
\begin{coro}\label{cor:classif}
Let $(M,g)$ be a connected, simply connected, oriented Riemannian manifold of dimension $n=2m+1$ rolling on $S^{2m+1}$. The following implications hold:
\begin{itemize}
\item if $H^{\nabla^{\rm R}}\subset{\rm U}(m+1)$, then $(M,g)$ is a Sasakian manifold,
\item if $H^{\nabla^{\rm R}}\subset{\rm SU}(m+1)$, then $(M,g)$ is a Sasaki-Einstein manifold, and
\item if $H^{\nabla^{\rm R}}\subset{\rm Sp}(k+1)$, if $n=4k+3$, then $(M,g)$ is a 3-Sasakian manifold.
\end{itemize}
\end{coro}

From Theorem~\ref{th:hol} it follows that the list in Theorem~\ref{th:MS} can be reduced using the classical Berger list of holonomy groups and dimension arguments. More precisely,

\begin{coro}\label{cor:posshol}
Let $(M,g)$ be a connected, simply connected, oriented Riemannian manifold of dimension $n=2m+1$ rolling on $S^{2m+1}$. The only possibilities of $H^{\nabla^{\rm R}}$ are the following
\begin{enumerate}
\item ${\rm SO}(n+1)$,

\item ${\rm U}(m+1)$,

\item ${\rm SU}(m+1)$,

\item ${\rm Sp}(k+1)$, if $n=4k+3$,

\item ${\rm Sp}(k+1)\cdot{\rm Sp}(1)$, if $n=4k+3$,


\item ${\rm Spin}(7)$, $n=7$.
\end{enumerate}
\end{coro}

\begin{proof}
From Theorem~\ref{th:hol}, we know that $H^{\nabla^{\rm R}}\cong H^{\tilde\nabla}$. The groups ${\rm Sp}(k)\cdot{\rm U}(1)$ and ${\rm Spin}(9)$, from the Montgomery-Samelson list do not appear in Berger's list of possible holonomy groups, and therefore can be removed. The group $G_2$ is removed by dimensionality reasons.
\end{proof}

In case (1), the rolling system is completely controllable. It follows from Table~\ref{table:hol} and Theorem~\ref{th:holcone} that in cases (2), (3) and (4) the Riemannian manifold $(M,g)$ is Sasakian, Sasaki-Einstein and 3-Sasakian, respectively. The cases (5) and (6) do not seem to have well-known names.

\section{Unitary rolling holonomy}\label{sec:unitary}

Let us start with the first non-trivial case in Corollary~\ref{cor:posshol}, that is when $H^{\nabla^{\rm R}}\subset {\rm U}(m+1)$. As was mentioned in Corollary~\ref{cor:classif} in this case the manifold $M$ is Sasakian. We want to show how the Sasakian structure on $M$ is inherited by rolling over the Sasakian unit sphere $S^{2m+1}$ by making use of the rolling connection instead of the K\"ahler structure of the Riemannian cone $C(M)$. 

Before stating and proving the main theorem in this section, let us recall equivalent definitions of Sasakian manifolds. See, for instance~\cite{BG1,Sas,S}. 

\begin{prop}\label{equvSas}
Let $(M,g)$ be a Riemannian manifold, with $\nabla$ the Levi-Civita connection of $g$ and $R(X,Y)$ the Riemannian curvature tensor. Then the following are equivalent:
\begin{itemize}
\item[a)]{There exists a Killing vector field $Z$ of unit length such that the tensor field $JX=\nabla_XZ$ satisfies
\begin{equation}
(\nabla_XJ)Y=g(Z,Y)X-g(X,Z)Y, \quad\text{for any}\quad X,Y\in {\rm VF}(M).
\end{equation}
}
\item[b)]{There exists a Killing vector field $Z$ of unit length such that the Riemannian curvature satisfies
\begin{equation*}
R(X,Z)Y=g(Z,Y)X-g(X,Z)Y, \quad\text{for any}\quad X,Y\in {\rm VF}(M).
\end{equation*}
}
\item[c)]{The Riemannian cone $C(M)=(M\times \mathbb R_+,\tilde g)$ is K\"ahler, i.e., $(M,g)$ is a Sasakian manifold.}
\end{itemize}
\end{prop}


\subsection{Unitary holonomy implies Sasakian} 

With all of these at hand, we have the following result.

\begin{theorem}\label{th:cont}
Let $(M,g)$ be a Riemannian manifold of dimension $2m+1$. Suppose that $H^{\nabla^{\rm R}}\subset {\rm U}(m+1)$. Then there exists a Killing vector field $Z$ of unit length and a $(1, 1)$ tensor field $J$ satisfying
\begin{equation}\label{Jcond}
JX=\nabla_XZ,\quad (\nabla_XJ)Y=g(Z,Y)X-g(X,Z)Y, 
\end{equation}
for any vectors $X,Y$ tangent to $M$.
Moreover, the one form $\alpha=g(Z,\cdot)$ is contact with Reeb vector field $Z$, the distribution 
$D=\ker\alpha$ is the contact distribution and $J|_D$ is a compatible almost complex structure on $D$. Thus $(Z,\alpha,J,g)$ is a Sasakian structure on $M$.
\end{theorem}

%

\begin{proof}
The scheme of the proof is as follows. Using the complex structure of $TM\oplus{\mathbb R}$ determined by the action of the group ${\rm U}(m+1)$, we define the $(1,1)$ tensor field $J$ and the vector field $Z$, and show that they satisfy conditions~\eqref{Jcond}. With these, it is possible to prove the rest of the theorem. For notational simplicity, in this proof we denote by $(X,r)$ the vector $X+r\partial_s|_{(x,s)}\in T_{(x,s)}C(M)$.

Let us fix an arbitrary point $x_0\in M$. Since $H^{\nabla^{\rm R}}|_{x_0}$ is a subgroup of the unitary group ${\rm U}(T_{x_0}M\oplus{\mathbb R})$, with respect to the metric $h=h_1$ defined in equation~\eqref{eq:hc}, it follows that there exists an $H^{\nabla^{\rm R}}|_{x_0}$-invariant almost complex structure 
\[
J_0^{\rm R}\colon T_{x_0}M\oplus{\mathbb R}\to T_{x_0}M\oplus{\mathbb R}.
\]
Specifically, the map $J_0^{\rm R}$ satisfies 
\begin{align*}
\|J_0^{\rm R}(X,r)\|_h&=\|(X,r)\|_h,&&\mbox{ for all }(X,r)\in T_{x_0}M\oplus{\mathbb R},\\
(J_0^{\rm R})^2&=-{\rm id},&&\\
BJ_0^{\rm R}&=J_0^{\rm R}B,&&\mbox{ for all }B\in H^{\nabla^{\rm R}}|_{x_0}\subset {\rm U}(T_{x_0}M\oplus{\mathbb R}).
\end{align*}
The parallel transport $\big(P^{\nabla^{\rm R}}\big)_0^1(\gamma)J_0^{\rm R}$ is independent of the curve $\gamma$ joining $x_0$ to $x$, and therefore the map
\[
J^{\rm R}_x:=\big(P^{\nabla^{\rm R}}\big)_0^1(\gamma)J_0^{\rm R},\quad\gamma(0)=x_0,\gamma(1)=x
\]
is a well-defined $(1,1)$-tensor field on $TM\oplus{\mathbb R}\to M$ satisfying
\begin{align*}
\|J^{\rm R}(X,r)\|_h&=\|(X,r)\|_h,\quad\mbox{ for all $(X,r)$ sections of $TM\oplus{\mathbb R}$,}\\
(J^{\rm R})^2&=-{\rm id},\\
\nabla^{\rm R}J^{\rm R}&=0.
\end{align*}
Let us define a $(1,1)$-tensor field $J$ on $M$, a one form $\beta\in\Omega^1(M)$ and a vector field $Z\in{\rm VF}(M)$ as follows
\begin{align*}
(J_xX,\beta_x(X))&:=J_x^{\rm R}(X,0),\quad X\in T_xM,\\
(Z_x,0)&:=J_x^{\rm R}(0,1),
\end{align*}
for all $x\in M$. To see that $J_x^{\rm R}(0,1)$ has second component identically zero, we simply notice that
\[
h(J_x^{\rm R}(0,1),(0,1))=0,
\]
for all $x\in M$. Let us check the properties of the $(1,1)$ tensor $J$.
Let $X,Y$ be sections of $D$. Then
\begin{equation}\label{eq:isom}
\begin{split}
g(JX,JY)=h\big((JX,0),(JY,0)\big)&=h\big(J^{\rm R}(X,0),J^{\rm R}(Y,0)\big)\\
&=h\big((X,0),(Y,0)\big)=g(X,Y).
\end{split}
\end{equation}
We also see that
\begin{equation}\label{eq:J2}
(-X,0)=(J^{\rm R})^2(X,0)=J^{\rm R}(JX,0)=(J^2X,0).
\end{equation}
An important fact to have in mind is that $JZ=0$, as can be seen directly
\begin{equation}\label{eq:JZ}
(0,-1)=(J^{\rm R})^2(0,1)=J^{\rm R}(Z,0)=(JZ,\beta(Z)).
\end{equation}

Let $X=X_D+aZ$ and $Y=Y_D+bZ$, where $X_D,Y_D$ are tangent vectors to $D$, and $a,b\in{\mathbb R}$. Combining the last two results we obtain
\begin{align}\label{eq:skewsym}
g(X,JY)&=g(X_D+aZ,J(Y_D+bZ))=g(X_D,J(Y_D))\\
&=g(JX_D,J^2Y_D)=-g(JX_D,Y_D)\nonumber\\
&=-g(J(X_D+aZ),Y_D+bZ)=-g(JX,Y).\nonumber
\end{align}

The next step is to show the property~\eqref{Jcond}.
Recall that the covariant derivative of a $(1,1)$-tensor $T$ is given by
\[
(\nabla_YT)(X)=\nabla_Y(T(X))-T(\nabla_YX),
\]
thus we can conclude that
\[
0=(\nabla_Y^{\rm R}J^{\rm R})(X,r)=\nabla^{\rm R}_Y(J^{\rm R}(X,r))-J^{\rm R}\nabla^{\rm R}_Y(X,r).
\]
Using the formula for $\nabla^{\rm R}$ from Definition~\ref{def:conn} and writing 
\[
J^{\rm R}(X,r)=J^{\rm R}(X,0)+rJ^{\rm R}(0,1)=(JX+rZ,\beta(X)).
\]
for $J^{\rm R}$, we see that
\begin{multline*}
J^{\rm R}\nabla^{\rm R}_Y(X,r)=J^R(\nabla_YX+rY,Y(r)-g(X,Y))\\
=(J(\nabla_YX+rY)+(Y(r)-g(X,Y))Z,\beta(\nabla_YX+rY)).
\end{multline*}
and similarly, we obtain
\begin{align*}
\nabla^{\rm R}_Y(J^{\rm R}(X,r))&=&&(\nabla_Y(JX+rZ)+\beta(X)Y,Y(\beta(X))-g(Y,JX+rZ))\\
&=&&((\nabla_YJ)(X)+J\nabla_YX+Y(r)Z+r\nabla_YZ+\beta(X)Y,\\
&&&Y(\beta(X))-g(Y,JX+rZ)).
\end{align*}

We obtain the equalities
\begin{align*}
rJY-g(X,Y)Z&=(\nabla_YJ)(X)+r\nabla_YZ+\beta(X)Y,\\
\beta(\nabla_YX)+r\beta(Y)&=Y(\beta(X))-g(Y,JX)-rg(Y,Z),
\end{align*}
valid for all $X,Y\in T_xM$ and all $r\in{\mathbb R}$. It follows that
\begin{align}
JY&=\nabla_YZ,\label{eq:JY}\\
-g(X,Y)Z&=(\nabla_YJ)(X)+\beta(X)Y,\label{eq:gJbeta}\\
\beta(\nabla_YX)&=Y(\beta(X))-g(Y,JX),\\
\beta(Y)&=-g(Y,Z)\label{eq:betag}.
\end{align}
Equation~\eqref{eq:JY} corresponds to the first part of~\eqref{Jcond}. Combining equations~\eqref{eq:gJbeta} and~\eqref{eq:betag} we obtain
\begin{equation*}
(\nabla_YJ)(X)=-g(X,Y)Z+g(X,Z)Y,
\end{equation*}
which corresponds to the second part of~\eqref{Jcond}. From the formulas above it follows easily that $Z$ is a Killing vector field, since
\[
g(\nabla_XZ,Y)+g(X,\nabla_YZ)=g(JX,Y)+g(X,JY)=0
\]
by~\eqref{eq:skewsym}. 

To finish the first part of the proof we need only to show that $Z$ has the length 1. Define the one form $\alpha(\cdot)=g(Z,\cdot)$, which obviously satisfies $\alpha=-\beta$ from equation~\eqref{eq:betag}. The property $\alpha(Z)=1$ from~\eqref{eq:JZ} and~\eqref{eq:betag} shows that $Z$ is a vector field of unit length. 

Recall that $D=\ker\alpha$ from the statement of the theorem. We proceed to prove that $J|_D$ is an endomorphism whose square equals $-{\rm id}|_D$, and that is an isometry with respect to $g|_D$. Let $X$ be a section of $D$, then
\[
J^{\rm R}(X,0)=(JX,-\alpha(X))=(JX,0),
\]
thus, since $J^{\rm R}$ is an isometry of $h$,
\begin{align*}
\alpha(JX)=g(Z,JX)&=h\big((Z,0),(JX,0)\big)\\
&=h\big(J^{\rm R}(0,1),J^{\rm R}(X,0)\big)=h\big((0,1),(X,0)\big)=0.
\end{align*}
It follows that $J|_D$ is an endomorphism. The property $(J|_D)^2=-{\rm id}$ and $
g(JX,JY)=g(X,Y)$ follows from~\eqref{eq:isom} and~\eqref{eq:J2}.

To conclude the proof of Theorem~\ref{th:cont}, we need to show that $(M,\alpha)$ is a contact manifold with Reeb vector field $Z$. Since we already showed that $\alpha(Z)=1$ in~\eqref{eq:JZ}, it suffice to show that $d\alpha(Z,\cdot)=0$. 
Define the map
\[
\omega(X,Y)=g(JX,Y),\quad X,Y\in T_xM.
\]
We want to show that $\omega$ is a two form satisfying ${\rm d}\alpha=2\omega$, and that $\omega|_D$ is non-degenerate. Let $X,Y\in T_xM$ and set $X_D=X-\alpha(X)Z, Y_D=Y-\alpha(Y)Z\in D_x$, then
\begin{align}\label{eq:2form}
\omega(X,Y)+\omega(Y,X)&=g(JX,Y)+g(JY,X)\nonumber\\
&=g(JX_D,Y_D+\alpha(Y)Z)+g(JY_D,X_D+\alpha(X)Z)\nonumber\\
&=g(JX_D,Y_D)+g(JY_D,X_D)=0,
\end{align}
which follows from the facts that $JZ=0$, the map $J|_D$ is an almost complex structure and an isometry with respect to $g|_D$. We conclude that $\omega$ is a two form. From Cartan's formula and properties of the Levi-Civita connection, we see that for any $X,Y\in{\rm VF}(M)$
\begin{align*}
{\rm d}\alpha(X,Y)&=X(\alpha(Y))-Y(\alpha(X))-\alpha([X,Y])\\
&=X(g(Z,Y))-Y(g(Z,X))-g(Z,[X,Y])\\
&=g(\nabla_XZ,Y)+g(Z,\nabla_XY)-g(\nabla_YZ,X)\\
&-g(Z,\nabla_YX)-g(Z,[X,Y])\\
&=g(JX,Y)-g(JY,X)+g(Z,\nabla_XY-\nabla_YX-[X,Y])\\
&=2\omega(X,Y).
\end{align*}
To see that $\omega$ is non-degenerate, pick a local orthonormal basis of $D$
\[
X_1,Y_1,X_2,Y_2,\dotsc,X_m,Y_m,\quad 2m+1=n
\]
such that $JX_i=Y_i$, $i\in\{1,\dotsc,m\}$. Then
\[
\omega(X_i,Y_j)=\delta_{i,j},\quad\omega(X_i,X_j)=\omega(Y_i,Y_j)=0,\quad i,j\in\{1,\dotsc,m\}.
\]
The non-degeneracy follows, and thus ${\rm d}\alpha$ is a symplectic form on $D$. It is well-know that this is equivalent to $\alpha$ being a contact form.

Finally, we note that
\[
{\rm d}\alpha(Z,\cdot)=2\omega(Z,\cdot)=2g(JZ,\cdot)=0,
\]
and since $\alpha(Z)=g(Z,Z)=1$, we conclude that $Z$ is the Reeb vector field of $\alpha$.
\end{proof}

\subsection{A converse result}

\begin{theorem}\label{th:convcont}
Let $(M,g)$ be a Riemannian manifold of dimension $2m + 1$. Suppose there exist a Killing vector field $Z\in{\rm VF}(M)$ of unit length and a $(1,1)$-tensor field $J$ satisfying the conditions in equation~\eqref{Jcond}, $JZ = 0$, and $J|_{ D}\colon{ D}\to{ D}$ an isometric almost complex structure on ${ D} = Z^\bot$. Then $H^{\nabla^{\rm R}} \subset {\rm U}(m+1)$.
\end{theorem}

\begin{proof}
Define a $(1,1)$-tensor field $J^{\rm R}$ on the vector bundle $TM\oplus{\mathbb R}\to M$ by
\begin{equation*}
J^{\rm R}(X, r) := (JX + rZ, -g(X, Z)).
\end{equation*}
We will show that $J^{\rm R}$ is both an isometry with respect to the metric $h=h_1$ in~\eqref{eq:hc}, and an almost complex structure on $TM\oplus{\mathbb R}\to M$ which is parallel with respect to the rolling connection $\nabla^{\rm R}$.

Indeed, if $(X,r)$ is a section of $TM \oplus {\mathbb R}$, we notice that $X_D = X -g(X,Z)Z$ is a section of $D$ satisfying $JX = JX_D$, and hence
\begin{align*}
\|J^{\rm R}(X,r)\|_h^2&=\|(JX + rZ, -g(X, Z))\|_h^2\\
&=\|JX_D\|_g^2+r^2\|Z\|_g^2+g(X,Z)^2\\
&=\|X_D\|_g^2+g(X,Z)^2+r^2=\|X\|_g^2+r^2=\|(X,r)\|_h^2,
\end{align*}
thus $J^{\rm R}$ is an isometry. To see that $J^{\rm R}$ is a complex structure on $TM\oplus{\mathbb R}$, we notice that
\begin{align*}
\big(J^{\rm R}\big)^2(X,r)&=J^{\rm R}(JX + rZ, -g(X, Z))\\
&=(J(JX + rZ)-g(X,Z)Z, -g(JX+rZ, Z))\\
&=(J^2X_D-g(X,Z)Z,-r)=(-X_D-g(X,Z)Z,-r)\\
&=-(X,r).
\end{align*}

To show that $J^{\rm R}$ is a parallel tensor field, we show that $\nabla^{\rm R}_Y(J^{\rm R}(X,r))=J^{\rm R}\nabla^{\rm R}_Y(X,r)$. This equality follows from
\begin{align*}
\nabla^{\rm R}_Y(J^{\rm R}(X,r))&=\nabla^{\rm R}_Y(JX + rZ, -g(X, Z))\\
&=(\nabla_Y (JX +rZ)-g(X,Z)Y,-Y(g(X,Z))-g(Y,JX +rZ))\\
&=((\nabla_Y J)X +J\nabla_Y X +Y(r)Z +r\nabla_Y Z -g(X,Z)Y,\\
&-g(\nabla_Y X,Z)-g(X,\nabla_Y Z)-g(Y,JX)-rg(Y,Z))\\
&=((\nabla_Y J)X +J\nabla_Y X +Y(r)Z +rJY -g(X,Z)Y,\\
&-g(\nabla_Y X,Z)-g(X,JY)-g(Y,JX)-rg(Y,Z))\\
&=(J(\nabla_Y X +rY)+(Y(r)-g(X,Y))Z,-g(\nabla_Y X +rY,Z))\\
&=J^{\rm R}(\nabla_Y X + rY, Y (r) - g(X, Y )) = J^{\rm R}\nabla^{\rm R}_Y (X, r),
\end{align*}
where we used the fact that $g(X,JY)=-g(JX,Y)$, whose proof is the same as in equation~\eqref{eq:skewsym}.

To conclude the proof, let us fix $x\in M$. The group ${\rm U}(m+1)$ can be identified with ${\rm U}(T_xM \oplus{\mathbb R}) := \{B \in {\rm SO}(T_xM\oplus{\mathbb R}) \,|\, [J^{\rm R}|_x,B] = 0\}.$
We know that $H^{\nabla^{\rm R}}|_x \subset {\rm SO}(T_xM\oplus{\mathbb R})$, so it is enough to notice that $\nabla^{\rm R}J^{\rm R} = 0$ means exactly that $[J^{\rm R}|_x,B] = 0$ for all $B\in H^{\nabla^{\rm R}}|_x$. Thus $H^{\nabla^{\rm R}}|_x \subset {\rm U}(T_xM \oplus{\mathbb R})$.
\end{proof}

\subsection{Example: Rolling the Heisenberg group on the sphere}

The $2m+1$ dimensional Heisenberg group ${\mathbf H}^m$ is the Lie group structure on ${\mathbb R}^{2m+1}$, whose Lie algebra is generated by the vector fields
\begin{equation}\label{eq:vfHn}
X_i=\partial_{x_i}-y_i\partial_z,\quad Y_i=\partial_{y_i}+x_i\partial_z,\quad Z=\partial_z,
\end{equation}
where $(x_1,y_1,\dotsc,x_n,y_n,z)$ are the standard coordinates in ${\mathbb R}^{2m+1}$. It is easy to see that $[X_i,Y_i]=2Z$, $i=1,\dotsc,m$, and all other Lie brackets vanish. We set a Riemannian metric $g$ on ${\mathbf H}^m$ for which the vector fields~\eqref{eq:vfHn} are orthonormal. 
From Koszul's formula for covariant derivatives, we can deduce that for $i=1,\dotsc,m$
\begin{align}
\nabla_{X_i}Y_i&=-\nabla_{Y_i}X_i=Z,\label{eq:covderH}\\
\nabla_{X_i}Z&=\nabla_ZX_i=-Y_i,\nonumber\\
\nabla_{Y_i}Z&=\nabla_ZY_i=X_i,\nonumber
\end{align}
and all other covariant derivatives vanish.

Following the ideas from previous sections, define a contact one form $\alpha(V)=g(V,Z)$ and a $(1,1)$-tensor $JV=\nabla_ZV$, for $V\in{\rm VF}({\mathbf H}^m)$. Note that formulas~\eqref{eq:covderH} imply that
\[
JX_i=-Y_i,\quad JY_i=X_i,\quad JZ=0. 
\]
It follows that $(Z,\alpha,J,g)$ is a Sasakian manifold. 

It is important to observe that the Riemannian cone $(C({\mathbf H}^m),\tilde g)$ of ${\mathbf H}^m$ is Ricci-flat if and only if the Ricci tensor of ${\mathbf H}^m$ satisfies the Einstein equation
\[
Ric_g=2mg,
\]
see~\cite[Proposition 1.9]{S}. Recall that, since $(C({\mathbf H}^m),\tilde g)$ is K\"ahler, then Ricci-flatness is equivalent to $H^{\tilde\nabla}\subset{\rm SU}(m+1)$. Since the Ricci tensor for ${\mathbf H}^m$ in the basis $\{X_1,Y_1,\dotsc,X_m,Y_m,Z\}$ is given by
\[
Ric_g=\begin{pmatrix}-2m&0&0&\cdots&0&0\\0&-2m&0&\cdots&0&0\\0&0&-2m&\cdots&0&0\\
\vdots&\vdots&\vdots&\ddots&\vdots&\vdots\\0&0&0&\cdots&-2m&0\\0&0&0&\cdots&0&2m\end{pmatrix},
\]
which is not a constant multiple of the identity matrix, we know that $C({\mathbf H}^m)$ is not Ricci-flat. By Theorem~\ref{th:hol}, it follows that $H^{\nabla^{\rm R}}\cong H^{\tilde\nabla}$ is a subgroup of ${\rm U}(m+1)$ in the list of Theorem~\ref{th:MS} which is not a subgroup of ${\rm SU}(m+1)$. We conclude that $H^{\nabla^{\rm R}}=U(m+1)$.

\subsection{Special unitary rolling holonomy}

It is known, see~\cite{S}, that simply connected Sasaki-Einstein manifolds are spin manifolds. It follows that a simply connected Sasaki-Einstein manifold of dimension three must be isometric to the standard 3-sphere. Thus, in dimension three, the rolling holonomy group ${\rm SU}(2)$ cannot occur. We conclude that for $M$ of dimension three, all non-controllable systems of $M$ rolling on $S^3$ have either trivial or ${\rm U}(2)$ rolling holonomy.

In a similar way, we can deduce a dichotomy as above for all Sasaki-Einstein manifolds of dimension $4m+1$.

\begin{prop}
Let $(M,g)$ be a Sasaki-Einstein manifold of dimension $4m+1$. Then the rolling holonomy group $H^{\nabla^{\rm R}}$ either coincides with ${\rm SU}(2m+1)$ or it is trivial.
\end{prop}

\begin{proof}
By dimension arguments, we see that the list in Corollary~\ref{cor:posshol} reduces to
\[
{\rm SO}(4m+2),\;{\rm U}(2m+1),\;{\rm SU}(2m+1).
\]
Since $M$ is Sasaki-Einstein, then $H^{\nabla^{\rm R}}$ must be a subgroup of ${\rm SU}(2m+1)$ in the above list, or trivial. The conclusion follows.
\end{proof}

The classification of Sasaki-Einstein manifolds of arbitrary dimension is still an open problem in differential geometry. A complete list of these in dimension five can be found in~\cite{BN}. 

\section{Symplectic rolling holonomy}

Following similar ideas to what was done in Section~\ref{sec:unitary}, we can prove with some modifications analogous results.

\begin{theorem}\label{th:3cont}
Let $(M,g)$ be a Riemannian manifold of dimension $4k+3$. Suppose that $H^{\nabla^{\rm R}}\subset {\rm Sp}(k+1)$. Then there exist three orthonormal Killing vector fields $Z_1,Z_2,Z_3$ and three $(1, 1)$ tensor fields $J_1,J_2,J_3$ satisfying
\begin{equation}\label{Jcond3Sasakian}
J_iX=\nabla_XZ_i,\quad (\nabla_XJ_i)Y=g(Z_i,Y)X-g(X,Z_i)Y, 
\end{equation}
for any vectors $X,Y$ tangent to $M$, and $[Z_i,Z_j]=2\epsilon_{i,j,k}Z_k$.
Moreover, for $i\in\{1,2,3\}$, the one forms $\alpha_i=g(Z_i,\cdot)$ are contact with Reeb vector fields $Z_i$, respectively. The distribution 
$D=\ker\alpha_1\cap\ker\alpha_2\cap\ker\alpha_3$ is the quaternionic contact distribution and $J_i|_D$ is a compatible almost complex structure on $D$. Thus $M$ is a 3-Sasakian manifold according to Definition~\ref{def:3sas}.
\end{theorem}

\begin{proof}
The construction of the characteristic vector fields $Z_i$ and the one forms $\alpha_i$, $i\in\{1,2,3\}$, is analogous to the one in Theorem~\ref{th:cont}. For completeness, we briefly recall this. The main technical concern is to show that 
\begin{equation}\label{eq:Zsu2}
[Z_i,Z_j]=2\epsilon_{i,j,k}Z_k.
\end{equation}

Since $H^{\nabla^{\rm R}}$ is a subgroup of the symplectic group ${\rm Sp}(k+1)$, with respect to the metric $h=h_1$ defined in equation~\eqref{eq:hc}, there exist three $H^{\nabla^{\rm R}}$-invariant almost complex structures $J_1^{\rm R}, J_2^{\rm R}, J_3^{\rm R}$, which are well-defined $(1,1)$-tensor fields on $TM\oplus{\mathbb R}\to M$ satisfying
\begin{align*}
\|J_i^{\rm R}(X,r)\|_h&=\|(X,r)\|_h,&&\mbox{ for all $(X,r)$ sections of $TM\oplus{\mathbb R}$,}\\
(J_i^{\rm R})^2&=-{\rm id},&&\\
\nabla^{\rm R}J_i^{\rm R}&=0,&&\\
J_i^{\rm R}J_j^{\rm R}&=-\epsilon_{i,j,k}J^{\rm R}_k,&&\{i,j,k\}=\{1,2,3\}.
\end{align*}
The last condition arises from the action of the symplectic group, see~\cite[Proposition 1.2.4]{BG1}.

As before, we define three $(1,1)$-tensor fields on $M$, three one forms $\alpha_i\in\Omega^1(M)$ and three vector fields $Z_i\in{\rm VF}(M)$ as follows
\begin{align*}
(J_iX,-\alpha_i(X))&:=J_i^{\rm R}(X,0),\quad X\in{\rm VF}(M),\\
(Z_i,0)&:=J_i^{\rm R}(0,1).
\end{align*}
All properties of these objects are proved in the same way as in Theorem~\ref{th:cont}. Now we proceed to verify~\eqref{eq:Zsu2}. We can easily verify the following relation
\begin{align*}
(J_iZ_j,-\alpha_i(Z_j))&=J_i^{\rm R}(Z_j,0)=J_i^{\rm R}J_j^{\rm R}(0,1)\\
&=-\epsilon_{i,j,k}J^{\rm R}_k(0,1)=-\epsilon_{i,j,k}(Z_k,0).
\end{align*}
We conclude that $J_iZ_j=-\epsilon_{i,j,k}Z_k$. Since $J_iZ_j=\nabla_{Z_j}Z_i$ and the Levi-Civita connection is torsion free, we have
\[
[Z_i,Z_j]=\nabla_{Z_i}Z_j-\nabla_{Z_j}Z_i=2\epsilon_{i,j,k}Z_k.\qedhere
\]
%
\end{proof}

And the converse also holds.

\begin{theorem}
Let $(M,g)$ be a Riemannian manifold of dimension $4k + 3$. Suppose there exist three orthonormal Killing vector fields $Z_1,Z_2,Z_3\in{\rm VF}(M)$ such that $[Z_i,Z_j]=2\epsilon_{i,j,k}Z_k$ and three $(1,1)$-tensor fields $J_1,J_2,J_3$ satisfying the conditions in equation~\eqref{Jcond3Sasakian}, $J_iZ_i = 0$, and $J_i|_{D_i}\colon{D_i}\to{D_i}$ are isometric almost complex structures on $D_i=Z_i^\bot$. Then $H^{\nabla^{\rm R}} \subset {\rm Sp}(k+1)$.
\end{theorem}

\begin{proof}
From~\cite[Proposition 1.2.2]{BG} and the hypothesis of this theorem, we see that $M$ has a 3-Sasakian structure with characteristic vector fields $Z_1,Z_2,Z_3\in{\rm VF}(M)$. We want to deduce that 
\begin{equation}\label{eq:Jcomm}
J_i^{\rm R}J_j^{\rm R}=-\epsilon_{i,j,k}J_k^{\rm R},\quad\{i,j,k\}=\{1,2,3\}.
\end{equation}
for the three $(1,1)$-tensor fields $J_i^{\rm R}$ on the vector bundle $TM\oplus{\mathbb R}\to M$ given by
\begin{equation*}
J_i^{\rm R}(X, r) := (J_iX + rZ_i, -g(X, Z_i)).
\end{equation*}
To do this, we first show that $J_iZ_j+J_jZ_i=0$. It is easy to compute that
\[
g(J_iZ_j+J_jZ_i,X)=g(J_iZ_j+J_jZ_i,Z_i)=g(J_iZ_j+J_jZ_i,Z_j)=0,
\]
for any $X$ a section of $D=D_1\cap D_2\cap D_3$. Thus, we only need to compute the following
\begin{align*}
g(J_iZ_j+J_jZ_i,Z_k)&=2\epsilon_{i,j,k}g(J_iZ_j+J_jZ_i,J_jZ_i-J_iZ_j)\\
&=2\epsilon_{i,j,k}\big(g(J_jZ_i,J_jZ_i)-g(J_iZ_j,J_iZ_j)\big)=0,
\end{align*}
because $J_i$ is an isometry on $D_i$. Since $[Z_i,Z_j]=J_jZ_i-J_iZ_j$, it follows that
\begin{equation}\label{eq:JandZ}
J_iZ_j=-\epsilon_{i,j,k}Z_k.
\end{equation}

To prove equation~\eqref{eq:Jcomm}, we divide the computation in three cases.
\begin{description}
\item[Cone vector field] From equation~\eqref{eq:JandZ} and the definition of $J_i^{\rm R}$, we have
\begin{align*}
J_i^{\rm R}J_j^{\rm R}(0,1)=&J_i^{\rm R}(Z_j,0)=(J_iZ_j,-\alpha_i(Z_j))\\
&=(-\epsilon_{i,j,k}Z_k,0)=-\epsilon_{i,j,k}J_k^{\rm R}(0,1).
\end{align*}

\item[$X$ section of $D$] Since $M$ is 3-Sasakian, we know that $J_iJ_j=-\epsilon_{i,j,k}J_k$ on $D$. Therefore
\begin{align*}
J_i^{\rm R}J_j^{\rm R}(X,0)=J_i^{\rm R}(J_jX,-\alpha_j(X))=(J_iJ_jX,0)=(-\epsilon_{i,j,k}J_kX,0)=-\epsilon_{i,j,k}J_k^{\rm R}(X,0).
\end{align*}

\item[$X$ in $M$ orthogonal to $D$] Let $X=a_iZ_i+a_jZ_j+a_kZ_k\in{\rm VF}(M)$ be orthogonal to $D$. From equation~\eqref{eq:JandZ}, we have
\begin{align*}
J_i^{\rm R}J_j^{\rm R}(X,0)&=\big(J_iJ_jX-\alpha_j(X)Z_i,-\alpha_i(J_jX)\big)=(a_iZ_j-a_jZ_i,\epsilon_{i,j,k}a_k)\\
-\epsilon_{i,j,k}J_k^{\rm R}(X,0)&=-\epsilon_{i,j,k}(-\epsilon_{i,j,k}a_iZ_j+\epsilon_{i,j,k}a_jZ_i,-a_k)=(a_iZ_j-a_jZ_i,\epsilon_{i,j,k}a_k).
\end{align*}
\end{description}

Recall that for any $x\in M$, the group ${\rm Sp}(k+1)$ can be identified with 
\[
{\rm Sp}(T_xM \oplus{\mathbb R}) := \{B \in {\rm SO}(T_xM\oplus{\mathbb R}) \,|\, [J_i^{\rm R}|_x,B] = 0,i\in\{1,2,3\}\},
\]
provided the endomorphisms $J_i^{\rm R}$ satisfy equation~\eqref{eq:Jcomm}. Thus $H^{\nabla^{\rm R}}|_x \subset {\rm Sp}(T_xM \oplus{\mathbb R})$.
\end{proof}

\end{document}